\documentclass[12pt,reqno]{amsart}
\usepackage{amsmath, amsthm, amscd, amsfonts, amssymb, xcolor}
\usepackage[bookmarksnumbered, colorlinks, plainpages]{hyperref}
\usepackage{color}	
\usepackage{appendix}
\usepackage{graphicx}
\usepackage{subfigure}
\usepackage{caption}
\usepackage{subcaption} 
\usepackage{todonotes}

\newcommand{\PB}{\overline{p}}
\newcommand{\CINF}{C_\infty}

\newcommand{\PAE}{\left( }
\newcommand{\PAD}{\right) }
\newcommand{\QUE}{\left[ }
\newcommand{\QUD}{\right] }
\newcommand{\CHE}{\left\{ }
\newcommand{\CHD}{\right\} }
\newcommand{\POE}{\left< }
\newcommand{\POD}{\right> }

\newcommand{\ULAPLA}{\Delta_1}
\newcommand{\PLAPLA}{\Delta_p}
\newcommand{\QLAPLA}{\Delta_q}
\newcommand{\INTO}{\int_\Omega}
\newcommand{\IINTO}{\displaystyle\int_\Omega}
\newcommand{\DIVE}{\mbox{div}}
\newcommand{\DX}{\hspace{0.05cm}\mbox{d}x}
\newcommand{\MOE}{\left| }
\newcommand{\MOD}{\right| }
\newcommand{\NOE}{\left\|  } 
\newcommand{\NOD}{\right\|  }

\newcommand{\NODWUPZ}{\right\|_{\scriptscriptstyle W^{1,p}_0}}
\newcommand{\NODWUQZ}{\right\|_{\scriptscriptstyle W^{1,q}_0}}

\newcommand{\NODLR}{\right\|_{\scriptscriptstyle L^r}}

\newcommand{\NODLQ}{\right\|_{\scriptscriptstyle L^q}}

\newcommand{\NODLINF}{\right\|_{\scriptscriptstyle L^\infty}}

\newcommand{\POmega}{(\Omega)}
\newcommand{\CR}{\mathbb{R}}

\newcommand{\ESPWUQZ}{W^{1,q}_0}
\newcommand{\ESPWUPZ}{W^{1,p}_0}
\newcommand{\ESPWUSZ}{W^{1,s}_0}

\newcommand{\ESPLP}{L^{p}}

\newcommand{\ESPLS}{L^{s}}
\newcommand{\ESPLR}{L^{r}}

\newcommand{\ESPLINF}{L^{\infty}}
\newcommand{\ESPCINFZ}{C^\infty_0}
\newcommand{\ESPCINFC}{C^\infty_c}
\newcommand{\ESPCUC}{C^1_c}

\newcommand{\ESPBV}{BV}

\newcommand{\MI}{\geq}
\newcommand{\mI}{\leq}

\newcommand{\IFF}{\Leftrightarrow}

\newcommand{\FUNUZ}{u_{\scriptscriptstyle 0}}
\newcommand{\FUNUB}{u_{\beta}}
\newcommand{\FUNUBZ}{u_{\beta_0}}
\newcommand{\FUNUPB}{u_{p,\beta}}
\newcommand{\RHOPB}{\rho_{p,\beta}}
\newcommand{\RHOB}{\rho_{\beta}}
\newcommand{\RHOZ}{\rho_{0}}
\newcommand{\FUNJPB}{\mathcal{J}_{p,\beta}}
\newcommand{\FUNFB}{\mathcal{F}_{\beta}}
\newcommand{\FUNPB}{\Phi_{p,\beta}}

\newcommand{\CAMZB}{\mathbf{z}_\beta}
\newcommand{\CAMZ}{\mathbf{z}}
\newcommand{\CAMZZ}{\mathbf{z}_{\scriptscriptstyle 0}}
\theoremstyle{definition}
\newtheorem{THEO}{Theorem}
\newtheorem{LEM}{Lemma}
\newtheorem{DEFIN}{Definition}
\newtheorem{PROP}{Proposition}
\begin{document}
\title[Existence of solutions $(1,q)$-Laplacian]{Existence of solutions for elliptic problems involving the $(1,q)$-Laplacian operator and a discontinuous superlinear nonlinearity}

\author[M. A. V. Costa]{Marcos A. V. Costa}
\author[O. H. Miyagaki]{Olímpio H. Miyagaki}
\author[M. T. O. Pimenta]{Marcos T. O. Pimenta}
	
\address[Marcos Antonio Viana Costa]
    {\newline\indent Universidade Federal de São Carlos - UFSCar
    \newline\indent Departamento de Matemática
	\newline\indent São Carlos -- SP -- Brazil}
	\email{\href{marcos.viana@ufscar.br}{marcos.viana@ufscar.br}}

\address[Olimpio Hiroshi Miyagaki]
    {\newline\indent Universidade Federal de São Carlos - UFSCar
    \newline\indent Departamento de Matemática
	\newline\indent São Carlos -- SP -- Brazil}
	\email{\href{olimpio@ufscar.br }{olimpio@ufscar.br }}

\address[Marcos Tadeu Oliveira Pimenta]
	{\newline\indent Universidade Estadual Paulista
	\newline\indent Departamento de Matemática e Computação
    \newline\indent Presidente Prudente -- SP -- Brazil}
	\email{\href{marcos.pimenta@unesp.br}{marcos.pimenta@unesp.br}}
\pretolerance10000
\begin{abstract}
\noindent  IIn this paper, we study a class of quasilinear elliptic problems involving the $(1,q)-$Laplacian operator and a discontinuous superlinear nonlinearity governed by the Heaviside function. The main difficulty of the problem arises from the presence of the $1$-Laplacian operator, whose natural setting is the Space of Functions of Bounded Variation. Our approach is based on an approximation method involving $(p,q)-$Laplacian problems as $p\to1^+$. As a consequence, we prove the existence of a nontrivial and nonnegative solution belonging to $\ESPWUPZ\POmega$, in an appropriate weak sense. Moreover, we investigate the asymptotic behavior of the solutions as $\beta\to0^+$, showing that the family of solutions converges to a solution of the limit problem without discontinuity.
\end{abstract}
\thanks{ Marcos A. V. Costa has been supported by FAPESP 2025/08701-5. Marcos T. O. Pimenta has been supported by FAPESP 2023/06617-1 and CNPq 303868/2024-4. Olimpio H. Miyagaki has been supported by FAPESP 2022/16407-1 and CNPq 303256/2022-2.}
\subjclass[2020]{26B30, 35B40, 35J60, 35J62}
\keywords{$(1,q)$-Laplacian Operator; Functions of Bounded Variation; Discontinuous nonlinearities}
\maketitle
\section{Introduction}
In this paper, we deal with the solutions of the following quasilinear elliptic problem involving the $(1,q)$-Laplacian
\begin{align}\label{P}\tag{P}
\begin{cases}\begin{array}{rll}
    -\ULAPLA u-\QLAPLA u=&\hspace*{-0.25cm} H(u-\beta)\MOE u\MOD^{s-2}u &\mbox{ in }\Omega,			\\
    u=&\hspace*{-0.25cm}0 &\mbox{ on }\partial\Omega,
\end{array}\end{cases}
\end{align}
where the $1$-Laplacian operator is formally defined by $\ULAPLA=\DIVE\PAE\frac{Du}{\MOE Du\MOD}\PAD$, $\Omega$ is a bounded domain with Lipschitz boundary in $\CR^N$, with $N\MI 2$, $\beta > 0$ is a real parameter, $1 < p < q < s<q^*=\frac{Nq}{N-q}$, and $H$ denotes the Heaviside function, which is defined as $H(t)=1$ when $t\MI 0$ and $H(t)=0$ otherwise.\newline

In the last three decades, problems involving the $1$-Laplacian operator have gained great interest and were subject of several research activities using different techniques. This operator, formally defined by
\begin{align*}
    \ULAPLA u:=\DIVE\PAE\dfrac{\nabla u}{\MOE\nabla u\MOD}\PAD,
\end{align*}
naturally leads to the study of functions of bounded variation, where the lack of reflexivity introduces significant technical challenges. For instance, Chata, Pimenta and Segura de León investigated in \cite{ChataPimentaSegura2023} the existence of multiple positive solutions for a concave–convex type problem involving the $1-$Laplacian by means of approximation techniques. In \cite{PimentaSantosStapenhorst2024}, Pimenta, Santos Júnior and Stapenhorst studied a quasilinear elliptic problem with discontinuous and critical nonlinearities, combining approximation methods with concentration-compactness arguments. Moreover, Pimenta, Carranza and Figueiredo in \cite{PimentaCarranzaFigueiredo2024} obtained existence results for problems with critical concave–convex nonlinearities via concentration-compactness techniques. Finally, Martínez Aparicio, Oliva and Petitta developed in \cite{AparicioOlivaPetitta2025} a sub-supersolution method for $1-$Laplacian type problems, highlighting additional difficulties related to comparison principles and compactness issues.

Several approaches have been developed to deal with such problems. One of the most widely used methods, (which was originally proposed in \cite{AndreuBallesterCasellesMazon1,AndreuBallesterCasellesMazon2,AndreuBallesterCasellesMazon3,Demengel}), consists of approximating the problem by means of $p$-Laplacian as $p\to1^+$, allowing the use of classical variational methods. This strategy has been successfully applied in various contexts, including in problems with power type nonlinearities (as in \cite{MolinoSegura2018}) and with discontinuous nonlinearities (as in \cite{PimentaSantosStapenhorst2024,SantosFigueiredoPimenta2022}).

More recently, Figueiredo, Pimenta and Winkert in \cite{FigueiredoPimentaWinkert2025} proposed the study of the $(p,q)-$Laplacian operator combining operators with different degrees of homogeneity. The limiting case $p\to 1^+$ leads to the $(1,q)-$Laplacian operator, whose analysis is even more delicate due to the interplay between the nonsmooth nature of the $1-$Laplacian and the regularizing effect of the $q-$Laplacian.

In recent years, a considerable amount of work has been devoted to elliptic problems involving the $(p,q)-$Laplacian operator, addressing issues such as existence, multiplicity of solutions, qualitative properties and asymptotic behavior. We refer, for instance, to the works \cite{ChavesErcoleMiyagaki2015,FariaMiyagakiMotreanu2014,MotreanuWinkert2019,PapageorgiouWinkert2021}, where different techniques have been successfully employed to treat superlinear and subcritical nonlinearities. However, despite these advances, problems involving the $(1,q)-$Laplacian setting remain largely unexplored.

Motivated by this context, and inspired by the problem proposed by Pimenta, Santos, and Santos Júnior in \cite{PimentaJunior2024}, the present work investigates an elliptic problem involving the $(1,q)-$Laplacian operator with a discontinuous Heaviside-type nonlinearity, contributing to the development of the theory by combining approximation methods, nonsmooth analysis, and the theory of divergence-measure fields.\newline

In this context, we aim to obtain solutions to the main problem of this paper in the following sense:
\begin{DEFIN}\label{DEFINSOL}
    We say that a function $u$ is a \textbf{bounded variation solution} to Problem \eqref{P} when there exist a function $\rho\in L^\frac{s}{s-1}\POmega$ and a vector field $\textbf{z}\in L^\infty(\Omega,\mathbb{R}^N)$, with $\NOE\textbf{z}\NODLINF\mI 1$, such that
    \begin{align}\begin{cases}
    -\DIVE\PAE\CAMZ\PAD-\QLAPLA u=\rho\hspace{1.5cm}\mbox{ in }\mathcal{D'}\POmega,\\
    \hspace{1.63cm}\CAMZ\cdot\nabla u=\MOE \nabla u\MOD\hspace{0.83cm}\mbox{ a.e. in }\Omega,\\
    \hspace{2.5cm}u=0\hspace{1.5cm}\mbox{ on }\partial\Omega,
    \end{cases}
    \end{align}
    Moreover, the function $\rho$ satisfies, for almost every $x\in\Omega$,
    \begin{align}
    \rho(x)\in
    \begin{cases}
        \{0\},                  &\mbox{ when }u(x)<\beta,\\
        [0,\beta^{s-1}],        &\mbox{ when }u(x)=\beta,\\
        \{u(x)^{s-1}\},   &\mbox{ when }u(x)>\beta.
    \end{cases}
\end{align}
\end{DEFIN}

In this way, we establish the following results by means of the method proposed by Figueiredo, Pimenta and Winkert in \cite{FigueiredoPimentaWinkert2025}:
\begin{THEO}\label{THEO1}
    Assume that $N\MI 2$ and $1<q<s<q^*$. For each $\beta>0$, Problem \eqref{P} admits at least one nonnegative and nontrivial solution $\FUNUB\in\ESPWUQZ\POmega$ in the sense of Definition \ref{DEFINSOL}.
\end{THEO}

Within the framework of the previous theorem, a natural and important question concerns the asymptotic behavior of the solutions $\FUNUB$ when $\beta\to0^+$. More precisely, it is reasonable to investigate whether the family $(\FUNUB)$ converges, in an appropriate sense, to a limit function as the parameter $\beta$ tends to zero through positive values. In fact, one expects that, when $\beta\to0^+$ the solutions $\FUNUB$ converge to a function $\FUNUZ$ that is a solution of the problem
\begin{align}\label{P0}\tag{P$_0$}
\begin{cases}\begin{array}{rll}
    -\ULAPLA u-\QLAPLA u=&\hspace*{-0.25cm}\MOE u\MOD^{s-2}u &\mbox{ in }\Omega,			\\
    u=&\hspace*{-0.25cm}0 &\mbox{ on }\partial\Omega.
\end{array}\end{cases}
\end{align}

The next theorem is devoted to rigorously establishing this convergence result, showing that the limit process indeed leads to a solution of Problem \eqref{P0} in an appropriate weak sense.
\begin{THEO}\label{THEO2}
    For each $\beta>0$, let $\FUNUB$ be the solution of Problem \eqref{P} obtained in Theorem \ref{THEO1}. There exists a nontrivial and nonnegative solution $\FUNUZ\in\ESPWUQZ\POmega$ of Problem \eqref{P0} such that, as $\beta\to0^+$,
    \begin{align*}
        \FUNUB&\to\FUNUZ\quad\hspace{0.42cm}\mbox{ in }L^r\POmega,\mbox{ for all }1\mI r<q^*,\\
        \FUNUB(x)&\to u(x)\quad\mbox{ a.e. in }\Omega.
    \end{align*}
    Moreover, there exist positive constants $\mu$ and $\beta_0$ such that
    \begin{align*}
        \MOE\{x\in\Omega;\ \FUNUB(x)>\beta\}\MOD\MI\mu
    \end{align*}
    for all $\beta\in(0,\beta_0)$, where $\MOE A\MOD$ denotes the Lebesgue measure of a measurable set $A\subset\CR^N$.
\end{THEO}

Although the problem involves the $1$-Laplacian operator, whose natural setting is the space $\ESPBV\POmega$, the obtained solution belongs to the Sobolev space $\ESPWUQZ\POmega$. This gain in regularity is a direct consequence of the presence of the $q$-Laplacian term, which exerts a regularizing effect on the solution.

Moreover, the formulation adopted for the notion of solution allows one to rigorously handle the discontinuity of the nonlinearity. In particular, the auxiliary function $\rho$ incorporates the multivalued nature of the Heaviside function at the point of discontinuity, which is essential to give meaning to the equation in a weak sense.\newline

This paper is organized as follows: In Section \ref{S02}, we present the necessary functional framework, including the Space of Functions of Bounded Variation, as well as fundamental tools from the Theory of Divergence-Measure Vector Fields and Nonsmooth Critical Point Theory. In Section \ref{S03}, we establish the existence of solutions to Problem \eqref{P}, proving Theorem \ref{THEO1} through an approximation method combined with variational techniques. Finally, in Section \ref{S04}, we analyze the asymptotic behavior of the solutions as $\beta\to0^+$ and prove Theorem \ref{THEO2}, showing the convergence to a solution of the Problem \eqref{P0}.
\section{Preliminaries}\label{S02}
In this section, we introduce the main functional framework and several auxiliary results that will be used throughout the paper. Since the operator involving the $1$-Laplacian naturally leads to gradients that are measures, the appropriate setting for our analysis is the space of functions of bounded variation. We also recall some basic tools from the theory of divergence-measure vector fields and nonsmooth critical point theory.

Throughout this work, $\Omega\subset\CR^N$, $N\MI2$, denotes a bounded open set with Lipschitz boundary $\partial\Omega$. We denote by $\mathcal{L}^N$ the Lebesgue measure in $\CR^N$ and by $\mathcal{H}^{N-1}$ the $(N-1)$-dimensional Hausdorff measure. The outward unit normal vector on $\partial\Omega$ is denoted by $\nu$.

We will use the usual Lebesgue and Sobolev spaces $\ESPLP\POmega$ and $W^{1,q}_0(\Omega)$ endowed with their standard norms.

First, let us observe that, by the boundedness of $\Omega$ and since $p<q$, it holds that
\begin{align*}
    \ESPWUPZ\POmega\cap\ESPWUQZ\POmega=\ESPWUQZ\POmega.
\end{align*}
Therefore, we consider as functional space the Sobolev space $\ESPWUQZ\POmega$, endowed with the norm
\begin{align*}
    \NOE u\NODWUQZ:=\PAE\INTO\MOE\nabla u\MOD^q\DX\PAD^\frac{1}{q}.
\end{align*}
Finally, we recall that $\ESPWUQZ\POmega$ is continuously embedded in $\ESPLR\POmega$ for all $1\mI r\mI q^*$, and compactly embedded in $\ESPLR\POmega$ for all $1\mI r< q^*$, where $q^*=Nq/(N-q)$.

\section*{Nonsmooth Critical Point Theory}
Since the nonlinear term appearing in the problem fail to be differentiable, we make use of tools from the critical point theory for locally Lipschitz functionals. This framework relies on the generalized gradient in the sense of Clarke \cite{Clarke1983,Clarke1975} and has been further developed in subsequent works by Chang \cite{Chang1981} and Grossinho and Tersian \cite{GrossinhoTersian2001}.

Let $E$ be a real Banach space and let $I:E\to\CR$ be a functional. We say that $I$ is \textbf{locally Lipschitz continuous} when, for all $u\in E$, there exist a neighborhood $V$ of $u$ and a constant $M>0$ such that
\begin{align*}
    \MOE I(v_2)-I(v_1)\MOD\mI\NOE v_2-v_1\NOD
\end{align*}
for all $v_1,v_2\in V$.

The \textbf{generalized directional derivative} of $I$ at $u$ in the direction $v\in E$ is defined by
\begin{align*}
    I^0(u,v):=\limsup\limits_{h\to 0,\ t\to0^+}\dfrac{I(u+h+tv)-I(u+h)}{t}.
\end{align*}
Moreover, $I^0(u,\cdot)$ is continuous, convex and its \textbf{subdifferential} at $z\in E$ is defined by
\begin{align*}
    \partial I^0(u,z):=\CHE\mu\in E^*;\ I^0(u,v)\MI I^0(u,z)+\POE\mu,v-z\POD,\mbox{ for all }v\in E\CHD,
\end{align*}
where $\POE\cdot,\cdot\POD$ is the duality pairing between $E^*$ and $E$.  The \textbf{generalized gradient} of $I$ at $u$ is the set
\begin{align*}
    \partial I(u):=\CHE\mu\in E^*;\ \POE\mu,v\POD\mI I^0(u,v),\mbox{ for all }v\in V\CHD,
\end{align*}
where this set convex, weak$^*$-compact and nonempty.

A point $u\in E$ is called a \textbf{critical point} of $I$ when $0\in\partial I(u)$ and, in this case, the number $c=I(u)$ is called a \textbf{critical value}.

To obtain critical points we use the nonsmooth version of the \textbf{Palais–Smale condition}. We say that $I$ satisfies the $(PS)_c$ condition if every sequence $(u_n)\subset E$ such that
\begin{align*}
    I(u_n)\to c\quad\mbox{ and }\quad\lambda_I(u_n)=\min\CHE\NOE\mu\NOD_{\scriptscriptstyle E^*};\ \mu\in\partial I(u_n)\CHD\to 0
\end{align*}
possesses a strongly convergent subsequence.

The following nonsmooth version of the Mountain Pass Theorem will be used later to obtain the existence of critical points for the functional associated with our problem.
\begin{THEO}{\textbf{(Nonsmooth Mountain Pass)}.} Let $E$ be a Banach Space and $I:E\to\CR$ be locally Lipschitz with $I(0)=0$. Suppose that there exist $\alpha>0$, $r>0$ and $e\in E$ such that
\begin{itemize}
    \item[(a)] $I(u)\MI\alpha$, when $\NOE u\NOD_{\scriptscriptstyle E}=r$.
    \item[(b)] $I(e)<0$ and $\NOE e\NOD_{\scriptscriptstyle E}>r$.
\end{itemize}
Define, for the set $\Gamma:=\CHE\gamma\in C([0,1]; E);\ \gamma(0)=0\mbox{ and }\gamma(1)=e\CHD$, the constant
\begin{align*}
    c=\inf\limits_{\gamma\in\Gamma}\max_{t\in[0,1]} I(\gamma(t)).
\end{align*}
Then $c\MI\alpha$ and there exists a sequence $(u_n)$ in $E$ such that
\begin{align*}
    I(u_n)\to c\quad\mbox{ and }\quad\lambda_I(u_n)\to c.
\end{align*}
If, in addition, $I$ satisfies the Palais–Smale condition, then $c$ is a critical value of $I$.
\end{THEO}

\section{Existence of Nonnegative Solution}\label{SECEXISSOL}\label{S03}
In this section we are going to prove Theorem \ref{THEO1}. As a first step, for $1 < p < q$, we consider the following auxiliary problem, which involves the $(p,q)$-Laplacian operator
\begin{align}\label{PP}\tag{P$_p$}
\begin{cases}\begin{array}{rll}
    -\PLAPLA u-\QLAPLA u=&\hspace*{-0.25cm} H(u-\beta)\MOE u\MOD^{s-2}u &\mbox{ in }\Omega,			\\
    u=&\hspace*{-0.25cm}0 &\mbox{ on }\partial\Omega,
\end{array}\end{cases}
\end{align}
whose weak formulation is given by the following integral identity
\begin{align}\label{FRACAPP}
    \INTO\MOE\nabla u\MOD^{p-2}\nabla u\cdot\nabla\varphi\DX+\INTO\MOE\nabla u\MOD^{q-2}\nabla u\cdot\nabla\varphi\DX=\INTO\RHOPB\varphi\DX
\end{align}
for all $\varphi\in\ESPWUQZ\POmega$, and it holds that, for almost every $x\in\Omega$,
\begin{align}\label{DEFINRHO}
    \RHOPB(x)\in
    \begin{cases}
        \{0\},                  &\mbox{ if }u(x)<\beta,\\
        [0,\beta^{s-1}],        &\mbox{ if }u(x)=\beta,\\
        \{u(x)^{s-1}\},   &\mbox{ if }u(x)>\beta.
    \end{cases}
\end{align}

Note that a functional $\FUNUPB\in\ESPWUQZ\POmega$ is a weak solution of Problem \eqref{PP} when there exists $\RHOPB\in L^{\frac{s}{s-1}}\POmega$ satisfying expressions \eqref{FRACAPP} and \eqref{DEFINRHO}.\\

Inspired by the arguments of Pimenta, Santos and Santos Júnior \cite{PimentaJunior2024}, who proved the existence of solutions in the case involving only the $1$-Laplacian operator, the critical points of the functional $\FUNJPB:\ESPWUQZ\POmega\to\CR$, given by
\begin{align}\label{FUNJPB}
    \FUNJPB(u):=\dfrac{1}{p}\INTO\MOE\nabla u\MOD^p\DX+\dfrac{1}{q}\INTO\MOE\nabla u\MOD^q-\FUNFB|_{\ESPWUSZ}(u)
\end{align}
in the sense of nonsmooth critical point theory, give rise to solutions of Problem \eqref{PP}, where $\FUNFB:\ESPLS\POmega\to\CR$ is given by
\begin{align*}
    \FUNFB(u)=\INTO F_\beta(u)\DX,
\end{align*}
with $f_\beta(t)=H(t-\beta)\MOE t\MOD^{s-2}t$ and $F_\beta(t_{0})=\displaystyle\int_0^{t_0}f_\beta(t)\DX$. Indeed, the functional $\FUNFB$ is locally Lipschitz and
\begin{align}\label{EXP34}
\partial\FUNFB(u)=\QUE\underline{f}_\beta(u),\overline{f}_\beta(u)\QUD
\end{align}
almost everywhere in $\Omega$, where
\begin{align*}
    \underline{f}_\beta(t)=\lim\limits_{r\to0^+}\mbox{ess}\inf\{f_\beta(s);\ \MOE t-s\MOD<r \}
\end{align*}
and
\begin{align*}
    \overline{f}_\beta(t)=\lim\limits_{r\to0^+}\mbox{ess}\sup\{f_\beta(s);\ \MOE t-s\MOD<r \}.
\end{align*}
The subdifferential of $\FUNFB$ takes the following values when analyzed with respect to the constant $\beta$
\begin{align}\label{EXP35}
    \QUE\underline{f}_\beta(u),\overline{f}_\beta(u)\QUD
    =\begin{cases}
        \{0\},          &\mbox{if }t<\beta,\\
        [0,\beta^{s-1}],&\mbox{if }t=\beta,\\
        \{t^{s-1}\},    &\mbox{if }t>\beta.
    \end{cases}
\end{align}

Due to the presence of the Heaviside function, the energy functional $\FUNJPB$ fails to be Fréchet differentiable. However, it is locally Lipschitz on $\ESPWUQZ\POmega$. Moreover, by \textit{Theorem 2.2} of \cite{Chang1981}, $\partial(\mathcal{F}_H|_{\ESPWUSZ})(u)=\partial\mathcal{F}_\beta(u)$, for all $u\in\ESPWUSZ\POmega$. Hence, by \textit{Proposition 2.1} of \cite{PimentaJunior2024}
\begin{align}\label{EXP37}
    \partial\FUNJPB(u)=\{Q'_p(u) \}-\partial\FUNFB(u)
\end{align}
for all $u\in\ESPWUSZ\POmega$, where $Q_p(u)=\dfrac{1}{p}\IINTO\MOE\nabla u\MOD^p\DX+\dfrac{1}{q}\IINTO\MOE\nabla u\MOD^q\DX$. Therefore, by expressions \eqref{EXP34}, \eqref{EXP35} and \eqref{EXP37}, critical points of $\FUNJPB$, in the sense of the nonsmooth critical point theory, will give rise to solutions of Problem \eqref{PP}.\\

In order to obtain a nontrivial solution of Problem \eqref{P}, we analyze the behavior of the solutions $\FUNUPB$ of Problem \eqref{PP} and show that, as $p\to 1^+$, they converge to a function $\FUNUB\in\ESPWUQZ\POmega$ which solves Problem \eqref{P} in a weak sense. 

In what follows, we will consider $p\in(1,\PB)$ for some $\PB\in(1,q)$ fixed, and the functional $\FUNPB:\ESPWUQZ\POmega\to\CR$ given by
\begin{align*}
    \FUNPB(u):=\FUNJPB(u)+\dfrac{p-1}{p}\MOE\Omega\MOD, 
\end{align*}
where a critical point of $\FUNPB$ implies a critical point of $\FUNJPB$, since these functionals differ only by a constant.

\begin{LEM}
\label{lemmabounded1}
    The functional $\FUNPB$ is nondecreasing with respect to $p$, that is, for all $u\in\ESPWUQZ\POmega$, $\Phi_{p_1,\beta}(u)\leq\Phi_{p_2,\beta}(u)$, where $1<p_1<p_2\mI q$.
\end{LEM}
\begin{proof}
    Note that, by Young’s Inequality for $1<p_1\mI p_2$,
    \begin{align*}
        \INTO\MOE\nabla u\MOD^{p_1}\DX\mI\dfrac{p_1}{p_2}\INTO\MOE\nabla u\MOD^{p_2}\DX+\dfrac{p_2-p_1}{p_2}\MOE\Omega\MOD.
    \end{align*}
    Given $u\in\ESPWUQZ\POmega$, we have
    \begin{align*}
        \Phi_{p_1,\beta}(u)
        &=\dfrac{1}{p_1}\INTO\MOE\nabla u\MOD^{p_1}\DX+\dfrac{1}{q}\INTO\MOE\nabla u\MOD^q-\FUNFB|_{\ESPWUSZ}(u)+\dfrac{p_1-1}{p_1}\MOE\Omega\MOD\\
        &\mI\dfrac{1}{p_2}\INTO\MOE\nabla u\MOD^{p_2}\DX+\dfrac{1}{q}\INTO\MOE\nabla u\MOD^q-\FUNFB|_{\ESPWUSZ}(u)+\dfrac{p_2-1}{p_2}\MOE\Omega\MOD\\
        &\mI\Phi_{p_2,\beta}(u),
    \end{align*}
    as we wished to prove.
\end{proof}

\begin{LEM}\label{LEMATPM}
    For each $p\in(1,\PB)$ and $\beta>0$, the functional $\FUNPB$ satisfies the geometric conditions of the Mountain Pass Theorem.
\end{LEM}
\begin{proof}
First, we must note that, by the definition of $\FUNFB$,
\begin{align*}
    F_\beta(t)\mI\dfrac{\MOE t\MOD^s}{s},
\end{align*}
for all $t\in\CR$. Consequently,
\begin{align*}
    \FUNPB(u)
    &=\dfrac{1}{p}\INTO\MOE\nabla u\MOD^p+\dfrac{1}{q}\INTO\MOE\nabla u\MOD^q-\INTO F_\beta(u)\DX+\dfrac{p-1}{p}\MOE\Omega\MOD\\
    &\MI\dfrac{1}{p}\INTO\MOE\nabla u\MOD^p+\dfrac{1}{q}\INTO\MOE\nabla u\MOD^q-\dfrac{1}{s}\INTO\MOE u\MOD^s\DX+\dfrac{p-1}{p}\MOE\Omega\MOD\\
    &\MI\dfrac{1}{q}\INTO\MOE\nabla u\MOD^q-\dfrac{1}{s}\INTO\MOE u\MOD^s\DX+\dfrac{p-1}{p}\MOE\Omega\MOD.
\end{align*}
Applying Hölder’s Inequality to $\dfrac{q^*}{s}$ and $\dfrac{q^*}{q^*-s}$, we obtain
\begin{align*}
    \INTO\MOE u\MOD^s\DX\mI\PAE\INTO\MOE u\MOD^{q^*}\DX\PAD^{\frac{s}{q^*}}\MOE\Omega\MOD^{\frac{q^*-s}{q^*}}=\NOE u\NOD_{\scriptscriptstyle L^{q^*}}^s\MOE\Omega\MOD^{\frac{q^*-s}{q^*}}.
\end{align*}
Moreover, using the argument employed in the proof of \textit{Theorem 7.10} in \cite{GilbargTrudinger1977}, we obtain
\begin{align}\label{EQ1}
    \NOE u\NOD_{\scriptscriptstyle L^{q^*}}\mI\dfrac{q(N-1)}{\sqrt{N}(N-1)}\NOE\nabla u\NODLQ,
\end{align}
for all $u\in\ESPWUQZ\POmega$. By combining the two inequalities, we obtain
\begin{align*}
    \FUNPB(u)
    &\MI\dfrac{1}{q}\NOE\nabla u\NODLQ^q-\dfrac{1}{s}\MOE\Omega\MOD^{\frac{q^*-s}{q^*}}\PAE\dfrac{q(N-1)}{\sqrt{N}(N-1)}\PAD^s\NOE\nabla u\NODLQ^s+\dfrac{p-1}{p}\MOE\Omega\MOD\\
    &\MI\dfrac{1}{q}\NOE\nabla u\NODLQ^q-C\NOE\nabla u\NODLQ^s+\dfrac{p-1}{p}\MOE\Omega\MOD\\
    &\MI\NOE\nabla u\NODLQ^s\PAE\dfrac{1}{q}\NOE \nabla u\NODLQ^{q-s}-C\PAD+\dfrac{p-1}{p}\MOE\Omega\MOD,
\end{align*}
where, in the last expression, we denote by $C=\dfrac{1}{s}\PAE\dfrac{q(N-1)}{\sqrt{N}(N-1)}\PAD^s\max\{1,\MOE\Omega\MOD\}$. 
Since $s>q$, we choose $\rho > 0$ sufficiently small so that
\begin{align*}
    \dfrac{1}{q}\rho^{q-s}-C > 0\IFF 0< \rho < \PAE\dfrac{1}{qC}\PAD^\frac{1}{s-q}.
\end{align*}
For such $\rho$ and $\displaystyle \alpha:=\rho\PAE\dfrac{1}{q}\rho^{q-s}-C\PAD + \frac{p-1}{p}|\Omega|$, it follows that $\Phi_{p,\beta}(u) \geq \alpha$, if $\|u\|_{W^{1,q}_0(\Omega)}$.

\noindent To verify the second geometric condition, we must consider $0 < \varphi\in\ESPCINFZ\POmega$ such that $\MOE A_{\varphi,\beta}\MOD>0$, where $A_{\varphi,\beta}:=\{x\in\Omega;\ \varphi(x)>\beta \} $. Note that, for each $t\MI 1$,
\begin{align*}
    \FUNPB(t\varphi)
    =\dfrac{t^p}{p}\INTO\MOE\nabla\varphi\MOD^p\DX+\dfrac{t^q}{q}\INTO\MOE\nabla\varphi\MOD^q\DX-\INTO F_\beta(t\varphi)\DX+\dfrac{p-1}{p}\MOE\Omega\MOD.
\end{align*}
In $A_{\varphi,\beta}$ we have that $t\varphi\MI\varphi>\beta$ and, consequently,
\begin{align*}
    \INTO\FUNFB(t\varphi)\DX\MI\int_{A_{\varphi,\beta}}\dfrac{(t\varphi)^s-\beta^s}{s}\DX=\dfrac{t^s}{s}\int_{A_{\varphi,\beta}}\varphi^s\DX-\dfrac{\beta^s}{s}\MOE A_{\varphi,\beta}\MOD.
\end{align*}
Thus
\begin{align*}
    \FUNPB(t\varphi)
    \mI&\dfrac{t^p}{p}\INTO\MOE\nabla\varphi\MOD^p\DX+\dfrac{t^q}{q}\INTO\MOE\nabla\varphi\MOD^q\DX-\dfrac{t^s}{s}\int_{A_{\varphi,\beta}}\varphi^s\DX\\
    &+\dfrac{\beta^s}{s}\MOE A_{\varphi,\beta}\MOD+\dfrac{p-1}{p}\MOE\Omega\MOD.
\end{align*}
Since $s>q$ and $s>p$, the term $-\frac{t^s}{s}\int_{A_{\varphi,\beta}}\varphi^s\DX$ dominates when $t\to+\infty$ and, consequently,
\begin{align*}
    \FUNPB(t\varphi)\to-\infty
\end{align*}
when $t\to+\infty$. Then, let $e=t\varphi$, such that $\Phi_{\overline{p},\beta}(e) < 0$ and note that, by Lemma \ref{lemmabounded1}, 
\begin{equation}
\label{eqMPT1}
\Phi_{p,\beta}(e) < 0, \quad \mbox{for all $1 < p \leq \overline{p}$.}
\end{equation}
Thus, the functional satisfies the second geometric condition of the Mountain Pass Theorem. 
\end{proof}

As in \cite{PimentaJunior2024}[Lemma 3.2], it is possible to prove the following result.

\begin{LEM}\label{LEMAPS}
    For each $p\in(1,\PB)$ and $\beta>0$, the functional $\FUNPB$ satisfies the nonsmooth Palais-Smale condition
\end{LEM}

As a consequence of lemmas \ref{LEMATPM} and \ref{LEMAPS}, we obtain, for each $p\in(1,\PB)$, a critical point $\FUNUPB\in\ESPWUPZ\POmega$ at level
\begin{align*}
    c_{p,\beta}=\inf\limits_{\gamma\in\Gamma}\max\limits_{t\in[0,1]}\FUNPB(\gamma(t))
\end{align*}
with $\Gamma_p=\{\gamma\in C([0,1],\ESPWUPZ\POmega);\ \gamma(0)=0\mbox{ and }\gamma(1)=e \}$, that is,
\begin{align}\label{FUNUPB}
    0\in\partial\FUNPB(\FUNUPB)
    \quad\mbox{ and }\quad
    \FUNPB(\FUNUPB)=c_{p,\beta}.
\end{align}

Moreover, there exists $\RHOPB\in L^{\frac{s}{s-1}}\POmega$ such that $\FUNUPB$ and $\RHOPB$ satisfy \eqref{FRACAPP} and \eqref{DEFINRHO}. Finally, by using $\varphi={u}_{p,\beta}^-$ as a test function in Expression \eqref{FRACAPP} and employing Expression \eqref{DEFINRHO}, we obtain $\NOE{u}_{p,\beta}^-\NODWUPZ^p=0$, which implies $\FUNUPB\MI 0$ almost everywhere in $\Omega$.

\begin{LEM}\label{LEM4}
    For each $p\in(1,\PB)$, the sequence $(\FUNPB(\FUNUPB))$ is non-decreasing.
\end{LEM}
\begin{proof}
    Let $1<p_1<p_2<\PB$. First, note that $\Gamma_{p_2}\subset\Gamma_{p_1}$, since, as \break $W^{1,{p_2}}_0\POmega\subset W^{1,{p_1}}_0\POmega$, the paths $\mathcal{C}\PAE[0,1],W^{1,{p_2}}_0\POmega\PAD$ also belong to $\mathcal{C}\PAE[0,1],W^{1,{p_1}}_0\POmega\PAD$. Hence, by Lemma \ref{lemmabounded1},
    \begin{align*}
        \Phi_{p_1,\beta}(u_{p_1,\beta})
        &=\inf\limits_{\gamma\in\Gamma_{p_1}}\max\limits_{t\in[0,1]}\Phi_{p_1,\beta}(\gamma(t))\\
        &\leq\inf\limits_{\gamma\in\Gamma_{p_2}}\max\limits_{t\in[0,1]}\Phi_{p_1,\beta}(\gamma(t))\\
        &\mI\inf\limits_{\gamma\in\Gamma_{p_2}}\max\limits_{t\in[0,1]}\Phi_{p_2,\beta}(\gamma(t))\\
        &\mI\Phi_{p_2,\beta}(u_{p_2,\beta})
    \end{align*}
    Therefore, the sequence $(\FUNPB(\FUNUPB))$ is non-decreasing.
\end{proof}

\begin{LEM}\label{LEM5}
    The sequence $(\FUNPB(\FUNUPB))_{p\in(1,\PB)}$ is bounded.
\end{LEM}
\begin{proof}
    By the monotonicity of the functional $\FUNPB$, it follows that, for a fixed $p_0\in(1,\PB)$,
    \begin{align*}
        \FUNPB(\FUNUPB)\mI\Phi_{p_0,\beta}(u_{p_0,\beta})=c_\beta
    \end{align*}
    for all $p\in(1,p_0)$, and consequently, the sequence $(\FUNPB(\FUNUPB))_{p\in(1,\PB)}$ is bounded. 
    
\end{proof}

\begin{LEM}\label{LEM6}
    The family $(\FUNUPB)_{p\in(1,\PB)}$ is bounded in $\ESPWUQZ\POmega$.
\end{LEM}
\begin{proof}
Note that, by Young’s Inequality and Lemma \ref{LEM5},
\begin{equation}
\label{eqbounded2}     c_{p,\beta}\mI c_{\PB,\beta}
\end{equation}
for all $p\in(1,\PB)$. Since $0\in\partial\FUNPB(\FUNUPB)$, there exists $\RHOPB$ such that
\begin{align*}
    Q^\prime_p(\FUNUPB)-\RHOPB=0.
\end{align*}
Taking the dual product of the above expression with $\FUNUPB$, we obtain
\begin{align}\label{EQ01LEM6}
    \INTO\MOE\nabla\FUNUPB\MOD^p\DX+\INTO\MOE\nabla\FUNUPB\MOD^q\DX=\INTO\RHOPB\FUNUPB\DX.
\end{align}
Note that, from \eqref{DEFINRHO} and the definition of $F_\beta$,
\begin{align}
\label{eqbounded31}
\INTO\PAE\dfrac{1}{s}\RHOPB\FUNUPB-F_\beta(\FUNUPB)\PAD\DX = \frac{\beta}{s}\int_{\{u_{p,\beta}=\beta\}} \rho_{p,\beta} \DX + \frac{\beta^s}{s}|\{u_{p,\beta} > \beta\}| \geq 0. 
\end{align}
Then
\begin{align*}
    c_{p,\beta}
    &=\FUNPB(\FUNUPB)-\dfrac{1}{s}\POE Q^\prime(\FUNUPB)-\RHOPB,\FUNUPB\POD\\
    &=\PAE\dfrac{1}{p}-\dfrac{1}{s}\PAD\INTO\MOE\nabla\FUNUPB\MOD^p\DX + \PAE\dfrac{1}{q}-\dfrac{1}{s}\PAD\INTO\MOE\nabla\FUNUPB\MOD^q\DX\\
    &\hspace{0.5cm}  + \INTO\PAE\dfrac{1}{s}\RHOPB\FUNUPB-F_\beta(\FUNUPB)\PAD\DX\\
    & \geq \PAE\dfrac{1}{p}-\dfrac{1}{s}\PAD\INTO\MOE\nabla\FUNUPB\MOD^p\DX + \PAE\dfrac{1}{q}-\dfrac{1}{s}\PAD\INTO\MOE\nabla\FUNUPB\MOD^q\DX.
\end{align*}
Hence, from Lemma \ref{LEM5}, there exists a constant $C=C(\beta) > 0$, such that
\begin{equation}
\label{eqbounded3}
\int_\Omega |\nabla u_{p,\beta}|^p \DX + \int_\Omega |\nabla u_{p,\beta}|^q \DX \leq C, \quad \mbox{for $1 < p < \overline{p}$}.
\end{equation}
from where it follows the result.
\end{proof}

The lemma above ensures that the family $(\FUNUPB)_{p\in(1,\PB)}$ is uniformly bounded in the space $W^{1,q}_0(\Omega)$, which allows passing to the limit as $p\to 1^+$. However, such an estimate does not provide pointwise control of the solutions. In order to obtain uniform bounds in $\ESPLINF\POmega$, which are essential for handling the nonlinear term and for subsequent analyses, the next result establishes, via the Moser Iteration technique, a uniform estimate in $\ESPLINF\POmega$, uniform on $p$. Its proof can be done by following the same lines as in \cite{PimentaJunior2024}[Lemma 3.4].

\begin{LEM}\label{LEM7}
    For each $\beta>0$, there exists $\CINF>0$ independent of $p\in(1,\PB)$ such that
    \begin{align}
        \NOE\FUNUPB\NODLINF\mI\CINF.
    \end{align}
\end{LEM}

As a consequence of Lemma \ref{LEM6}, the family $(\FUNUPB)_{p\in(1,\PB)}$ is uniformly bounded in $\ESPWUQZ\POmega$. Since the embedding $\ESPWUQZ\POmega\hookrightarrow L^r\POmega$ is compact for every $r\in[q,q^*)$, there exists a function $\FUNUB\in\ESPWUQZ\POmega$ such that, as $p\to1^+$,
\begin{align}
    \label{CONVER1} \FUNUPB&\rightharpoonup\FUNUB\quad\mbox{ in }\ESPWUQZ\POmega,\\
    \label{E326} \FUNUPB&\to\FUNUB\quad\mbox{ in }L^r\POmega,\mbox{ with }1 \leq r < q^*,\\
    \label{E327} \FUNUPB&\to\FUNUB\quad\mbox{ a.e. in } \Omega.
\end{align}

Moreover, the limit function satisfies $\FUNUB\in\ESPLINF\POmega$ and $\FUNUB(x)\MI 0$ for almost every $x\in\Omega$.\\

The next step in the proof of Theorem \ref{THEO1} is to show that the limit function $\FUNUB$, obtained by passing to the limit as $p\to 1^+$, is indeed a solution of Problem \eqref{P} in the sense of Definition \ref{DEFINSOL}. To achieve this, we must carefully analyze the behavior of the approximating solutions $\FUNUPB$ and of the terms associated with the discontinuous nonlinearity in the limit $p \to 1^+$. In particular, we need to identify the appropriate limit of the sequence $\RHOPB$.

\begin{LEM}\label{LEM8}
    Assume that $\FUNUPB\in\ESPWUQZ\POmega$, $\RHOPB\in L^{\frac{s}{s-1}}\POmega$, and $\FUNUB\in\ESPBV\POmega$ satisfy \eqref{FRACAPP}, \eqref{DEFINRHO}, and \eqref{E326}, respectively. There exists $\RHOB\in L^{\frac{s}{s-1}}\POmega$ such that
    \begin{align}\label{E328}
        \RHOPB\rightharpoonup\RHOB\quad\mbox{ in }L^{\frac{s}{s-1}}\POmega,
    \end{align}
    as $p\to1^+$. Moreover, $\RHOB$ satisfies, for almost every $x\in\Omega$,
    \begin{align}\label{E329}
    \RHOB(x)\in
    \begin{cases}
        \{0\},                  &\mbox{ if }\FUNUB(x)<\beta,\\
        [0,\beta^{s-1}],        &\mbox{ if }\FUNUB(x)=\beta,\\
        \{\FUNUB(x)^{s-1}\},   &\mbox{ if }\FUNUB(x)>\beta.
    \end{cases}
\end{align}
\end{LEM}
\begin{proof}
See \textit{Lemma 3.5} of \cite{PimentaJunior2024}.
\end{proof}

Through the previous lemmas, we have all the necessary framework to analyze and conclude that, indeed, the conditions of Definition \ref{DEFINSOL} are satisfied for each $\beta>0$.

\begin{LEM}
For each $\beta>0$, let $\FUNUPB$ be a solution of Problem \eqref{PP}. There exists a vector field $\CAMZB$, with $\NOE\CAMZB\NODLINF\mI1$, such that, up to a subsequence,
\begin{align}\label{E345}
    \MOE\nabla\FUNUPB\MOD^{p-2}\nabla\FUNUPB\rightharpoonup\CAMZB
\end{align}
in $L^r\PAE\Omega,\CR^N\PAD$ for all $1\mI r<\infty$, as $p\to1^+$. Moreover
\begin{align}\label{eqweakpq}
    \MOE\nabla\FUNUPB\MOD^{q-2}\nabla\FUNUPB\rightharpoonup\MOE\nabla\FUNUB\MOD^{q-2}\nabla\FUNUB
\end{align}
in $L^\frac{q}{q-1}(\Omega,\CR^N)$, as $p\to1^+$.
\end{LEM}
\begin{proof}
Let us first show that $\PAE \MOE\nabla\FUNUPB\MOD^{p-2}\nabla\FUNUPB\PAD_{p>1}$ is bounded in in $L^r\PAE\Omega,\CR^N\PAD$, with $1\mI r<\infty$. Given $1\mI r<\infty$, consider $p > 1$ sufficiently close to $1$, such that $r<\frac{p}{p-1}$. Then
\begin{align*}
    \NOE\MOE\nabla\FUNUPB\MOD^{p-2}\nabla\FUNUPB\NODLR^r
    &=\INTO\MOE\MOE\nabla\FUNUPB\MOD^{p-2}\nabla\FUNUPB\MOD^r\DX\\
    &=\INTO\MOE\nabla\FUNUPB\MOD^{(p-1)r}\DX.
\end{align*}
By applying Hölder’s Inequality with exponents $\frac{p}{(p-1)r}$ and $\frac{p}{p-(p-1)r}$, we obtain
\begin{align*}
    \INTO\MOE\nabla\FUNUPB\MOD^{(p-1)r}\DX\mI\PAE\INTO\MOE\nabla \FUNUPB\MOD^p\DX\PAD^{\frac{(p-1)r}{p}}\MOE\Omega\MOD^{1-\frac{(p-1)r}{p}}.
\end{align*}
Thus, by the uniform control obtained in Expression \eqref{eqbounded3}, it follows that
\begin{align}\label{ELIM}
    \NOE\MOE\nabla\FUNUPB\MOD^{p-2}\nabla\FUNUPB\NODLR\mI C^{\frac{(p-1)}{p}}\MOE\Omega\MOD^{\frac{1}{r}-\frac{(p-1)}{p}}
\end{align}
and consequently, the sequence $\PAE \MOE\nabla\FUNUPB\MOD^{p-2}\nabla\FUNUPB\PAD_{p>1}$ is bounded in $L^r\PAE\Omega,\CR^N\PAD$, with $1\mI r<\infty$. By weak compactness in $L^r\PAE\Omega,\CR^N\PAD$, there exists $\mathbf{z}_{\beta,r}$ such that, up to a subsequence,
\begin{align*}
    \MOE\nabla\FUNUPB\MOD^{p-2}\nabla\FUNUPB\rightharpoonup\mathbf{z}_{\beta,r}
\end{align*}
in $L^r\PAE\Omega,\CR^N\PAD$. Using a diagonal argument, it is possible to show that the limit $\mathbf{z}_{\beta,r}$ does not depend on $r$, that is, as $p \to 1^+$,
\begin{align*}
    \MOE\nabla\FUNUPB\MOD^{p-2}\nabla\FUNUPB\rightharpoonup\CAMZB
\end{align*}
in $L^r\PAE\Omega,\CR^N\PAD$ for all $1\mI r<\infty$. Moreover, using the fact that the norm in $L^r$ is weakly lower semicontinuous and the inequality obtained in Expression \eqref{ELIM}, we obtain
\begin{align*}
    \NOE\CAMZB\NODLINF
    &\mI\liminf\limits_{p\to1^+}\NOE \MOE\nabla\FUNUPB\MOD^{p-2}\nabla\FUNUPB\NODLR\\
    &\mI\liminf\limits_{p\to1^+}\PAE C^{\frac{(p-1)}{p}}\MOE\Omega\MOD^{\frac{1}{r}-\frac{(p-1)}{p}}\PAD\\
    &= \MOE\Omega\MOD^{\frac{1}{r}}
\end{align*}
for all $1\mI r<\infty$. Consequently, by letting $r\to+\infty$, it follows that $\CAMZB\in\ESPLINF(\Omega,\CR^N)$ and $\NOE\CAMZB\NODLINF\mI 1$. 

Moreover, from \cite{BocardoMurat}[Theorem 2.1], 
\begin{align*}
    \nabla\FUNUPB(x)\to\nabla\FUNUB(x)
\end{align*}
for almost every $x\in\Omega$, as $p\to1^+$. Moreover, by the Expression \eqref{eqbounded3}, the sequence $(\MOE\nabla\FUNUPB\MOD^{q-2}\nabla\FUNUPB)_{p\in(1,\PB)}$ is bounded in $L^{\frac{q}{q-1}}(\Omega,\CR^N)$. From \cite{Kavian}[Theorem 4.8], it follows that
\begin{align*}
    \MOE\nabla\FUNUPB\MOD^{q-2}\nabla\FUNUPB\rightharpoonup\MOE\nabla\FUNUB\MOD^{q-2}\nabla\FUNUB
\end{align*}
in $L^\frac{q}{q-1}(\Omega,\CR^N)$, as $p\to1^+$.
\end{proof}

It is important to observe that, from the weak convergence given in Expression \eqref{E345}, we obtain
\begin{align}\label{E346}
    -\PLAPLA\FUNUPB=\DIVE\PAE\MOE\nabla\FUNUPB\MOD^{p-2}\nabla\FUNUPB\PAD\to\DIVE(\CAMZB)\quad\mbox{ in }\mathcal{D}^\prime\POmega,
\end{align}
as $p \to 1^+$. Moreover, by expressions \eqref{FRACAPP}, \eqref{E326}, \eqref{E328} and \eqref{eqweakpq}, applying the Lebesgue Dominated Convergence Theorem, we conclude that
\begin{align}\label{COND1}
    -\DIVE(\CAMZB)-\QLAPLA\FUNUB=\RHOB\quad\mbox{ in }\mathcal{D}^\prime\POmega,
\end{align}
with $\RHOB$ satisfying the Expression \eqref{E329}. Consequently, this expression provides the following weak formulation in the sense of distributions:
\begin{align}\label{EQ001}
    \INTO\CAMZB\cdot\nabla\varphi\DX+\INTO\MOE\nabla\FUNUB\MOD^{q-2}\nabla\FUNUB\cdot\nabla\varphi\DX=\INTO\RHOB\varphi\DX
\end{align}
for all $\varphi\in\ESPCINFZ\POmega$. 

\begin{LEM}\label{LEM10}
For each $\beta>0$, the function $\FUNUB$ and the vector field $\CAMZB$ satisfy
\begin{align}\label{COND2}
    \CAMZB\cdot\nabla\FUNUB=\MOE\nabla\FUNUB\MOD\quad\mbox{ a.e. in }\Omega.
\end{align}
\end{LEM}
\begin{proof}
By Lemma \ref{LEM8} we know that $\NOE\CAMZB\NODLINF\mI 1$. Hence, for almost every $x\in\Omega$
\begin{align*}
    \CAMZB(x)\cdot\nabla\FUNUB(x)\mI\MOE\CAMZB(x)\MOD\MOE\nabla\FUNUB(x)\MOD\mI\MOE\nabla\FUNUB(x)\MOD
\end{align*}
and, consequently,
\begin{align}\label{E01LEM8}
    \CAMZB\cdot\nabla\FUNUB\mI\MOE\nabla\FUNUB\MOD\quad\mbox{ a.e. in }\Omega.
\end{align}
We now prove the reverse inequality. It suffices to show that for every nonnegative $\varphi\in\ESPCUC\POmega$
\begin{align}
    \INTO\varphi\CAMZB\cdot\nabla\FUNUB\DX\MI\INTO\varphi\MOE\nabla\FUNUB\MOD\DX.
\end{align}
Since $\FUNUB$ may not be smooth enough, we employ mollifiers. Let $(\eta_\varepsilon)_{\varepsilon>0}$ be a family of mollifiers. Taking $(\FUNUB\varphi)*\eta_\varepsilon$ as a test function in \eqref{COND1}, we obtain
\begin{align*}
    \INTO\CAMZB\cdot\nabla(\FUNUB\varphi*\eta_\varepsilon)\DX=-\INTO\MOE\nabla\FUNUB\MOD^{q-2}\nabla\FUNUB\cdot\nabla(\FUNUB\varphi*\eta_\varepsilon)\DX+\INTO\RHOB(\FUNUB\varphi*\eta_\varepsilon)\DX
\end{align*}
and, consequently,
\begin{align*}
    \INTO\CAMZB\cdot\nabla\FUNUB\varphi*\eta_\varepsilon\DX
    =&-\INTO\MOE\nabla\FUNUB\MOD^q\varphi*\eta_\varepsilon\DX-\INTO\MOE\nabla\FUNUB\MOD^{q-2}\nabla\FUNUB\cdot\nabla\varphi(\FUNUB*\eta_\varepsilon)\DX\\
    &+\INTO\RHOB(\FUNUB\varphi*\eta_\varepsilon)\DX-\INTO\CAMZB\cdot\nabla\varphi(\FUNUB*\eta_\varepsilon)\DX.
\end{align*}
Passing to the limit as $\varepsilon\to0^+$, we obtain
\begin{align}\label{E01LEM10}
    \nonumber\INTO\CAMZB\cdot\nabla\FUNUB\varphi\DX
    =&-\INTO\varphi\MOE\nabla\FUNUB\MOD^q\DX-\INTO\FUNUB\MOE\nabla\FUNUB\MOD^{q-2}\nabla\FUNUB\cdot\nabla\varphi\DX\\
    &+\INTO\RHOB\FUNUB\varphi\DX-\INTO\FUNUB\CAMZB\cdot\nabla\varphi\DX.
\end{align}
On the other hand, taking $\FUNUPB\varphi\in\ESPWUQZ\POmega$ as a test function in Problem \eqref{PP}, we obtain, after rearranging the terms,
\begin{align}\label{E02LEM10}
    \nonumber\INTO\MOE\nabla\FUNUPB\MOD^p\varphi\DX+\INTO\MOE\nabla\FUNUPB\MOD^q\varphi\DX+\INTO\FUNUPB\MOE\nabla\FUNUPB\MOD^{p-2}\nabla\FUNUPB\cdot\nabla\varphi\DX\\+\INTO\FUNUPB\MOE\nabla\FUNUPB\MOD^{q-2}\nabla\FUNUPB\cdot\nabla\varphi\DX=\INTO\RHOPB\FUNUPB\varphi\DX.
\end{align}
Passing to the limit as $p\to1^+$ and using \eqref{E326}, \eqref{E328} and \eqref{E345}, we obtain
\begin{align}\label{E03LEM10}
    \INTO\FUNUPB\MOE\nabla\FUNUPB\MOD^{p-2}\nabla\FUNUPB\cdot\nabla\varphi\DX\to\INTO\FUNUB\CAMZB\cdot\nabla\varphi\DX
\end{align}
and
\begin{align}\label{E04LEM10}
    \INTO\RHOPB\FUNUPB\varphi\DX \to \INTO \rho_\beta u_\beta \varphi\DX.
\end{align}
Finally, by Young's Inequality,
\begin{align*}
    \INTO\varphi\MOE\nabla\FUNUPB\MOD\DX\mI\dfrac{1}{p}\INTO\varphi\MOE\nabla\FUNUPB\MOD^{p}+\dfrac{p-1}{p}\INTO\varphi\DX
\end{align*}
and, taking the lower limit as $p\to1^+$,
\begin{align}\label{E05LEM10}
    \INTO\varphi\MOE\nabla\FUNUPB\MOD\mI\liminf\limits_{p\to1^+}\INTO\varphi\MOE\nabla\FUNUPB\MOD^{p}.
\end{align}
Combining the expressions \eqref{CONVER1}, \eqref{E01LEM10}, \eqref{E02LEM10}, \eqref{E03LEM10}, \eqref{E04LEM10} and \eqref{E05LEM10},
\begin{align*}
    \INTO\varphi\CAMZB\cdot\nabla\FUNUB\MI\liminf\limits_{p\to1^+}\INTO\varphi\MOE\nabla\FUNUPB\MOD^p\DX\MI\INTO\varphi\MOE\nabla\FUNUB\MOD\DX,
\end{align*}
which proves the reverse inequality and completes the proof.
\end{proof}

So far, we have shown that, for each $\beta>0$, there exists $\FUNUB\in\ESPWUQZ\POmega$ for which there is a vector field $\CAMZB\in L^\infty(\Omega,\mathbb{R}^N)$, with $\NOE\CAMZB\NODLINF\mI 1$, such that
\begin{align}\begin{cases}\label{HIPZB}
    -\DIVE\PAE\CAMZB\PAD-\QLAPLA\FUNUB=\RHOB\hspace{1.5cm}\mbox{ in }\mathcal{D'}\POmega,\\
    \hspace{1.64cm}\CAMZB\cdot\nabla\FUNUB=\MOE\nabla\FUNUB\MOD\hspace{0.83cm}\mbox{ a.e. in }\Omega,\\
    \hspace{2.7cm}\FUNUB=0\hspace{1.75cm}\mbox{ on }\partial\Omega,
\end{cases}
\end{align}
where the function $\RHOB$ satisfies, for almost every $x\in\Omega$,
\begin{align}\label{HIPRB}
    \RHOB(x)\in
\begin{cases}
        \{0\},                  &\mbox{ when }\FUNUB(x)<\beta,\\
        [0,\beta^{s-1}],        &\mbox{ when }\FUNUB(x)=\beta,\\
        \{\FUNUB(x)^{s-1}\},   &\mbox{ when }\FUNUB(x)>\beta.
\end{cases}
\end{align}

It remains to show that the solution we obtained is nontrivial, which is ensured by the following proposition:
\begin{PROP}
   For each $\beta>0$, the solution $\FUNUB\in\ESPWUQZ\POmega$ is nontrivial.
\end{PROP}
\begin{proof}
Observe that $\FUNUPB\MI 0$ for almost every $x\in\Omega$, and, by Expression \eqref{E327} and Lemma \ref{LEM7}, $\FUNUB\in\ESPLINF\POmega$ and $\FUNUB(x)\MI 0$ for almost every $x\in\Omega$. Moreover, by Lemma \ref{LEMATPM} and Expression \eqref{FUNUPB}, from the minimax level $c_{p,\beta}$, there exists a constant $\alpha>0$, independent of $p\in(1,\PB)$, such that
\begin{align}\label{NTA}
    \alpha+o_p(1)\mI c_{p,\beta}
\end{align}
when $p\to1^+$. Taking $\varphi=\FUNUPB$ in the weak formulation of Problem \eqref{PP}, we obtain
\begin{align}\label{NTB}
    \INTO\MOE\nabla\FUNUPB\MOD^p\DX+\INTO\MOE\nabla\FUNUPB\MOD^q\DX=\INTO\RHOPB\FUNUPB\DX.
\end{align}
In a similar way to what was done in Lemma \ref{LEM6}, it follows that
\begin{align*}
    \FUNPB(\FUNUPB)-\dfrac{1}{s}\POE\mu_{p,\beta},\FUNUPB\POD
    =&\PAE\dfrac{1}{p}-\dfrac{1}{s}\PAD\INTO\MOE\nabla\FUNUPB\MOD^p\DX+\PAE\dfrac{1}{q}-\dfrac{1}{s}\PAD\INTO\MOE\nabla\FUNUPB\MOD^q\DX\\
    &+\INTO\PAE\dfrac{1}{s}\RHOPB\FUNUPB-F_\beta(\FUNUPB)\PAD\DX+\dfrac{p-1}{p}\MOE\Omega\MOD,
\end{align*}
where $\mu_{p,\beta}\in\partial\FUNPB(\FUNUPB)$. Moreover, by taking $u_{p,\beta}$ as test function in the weak form of \eqref{PP}, we obtain the estimate
\begin{align}\label{NTE}
    c_{p,\beta}=\FUNPB(\FUNUPB)\mI\dfrac{1}{p}\INTO\MOE\nabla\FUNUPB\MOD^p\DX+\dfrac{1}{q}\INTO\MOE\nabla\FUNUPB\MOD^q\DX+o_p(1).
\end{align}
Combining the expressions \eqref{NTA} and \eqref{NTE}, we get
\begin{align*}
    \alpha\mI\liminf\limits_{p\to1^+}\PAE\dfrac{1}{p}\INTO\MOE\nabla\FUNUPB\MOD^p\DX+\dfrac{1}{q}\INTO\MOE\nabla\FUNUPB\MOD^q\DX\PAD
\end{align*}
and since $\frac{1}{p},\frac{1}{q}\mI1$, by Expression \eqref{NTB}, follows that
\begin{align*}
    \dfrac{1}{p}\INTO\MOE\nabla\FUNUPB\MOD^p\DX+\dfrac{1}{q}\INTO\MOE\nabla\FUNUPB\MOD^q\DX\mI\INTO\RHOPB\FUNUPB\DX.
\end{align*}
Consequently,
\begin{align}\label{NTG}
    0<\alpha\mI\liminf\limits_{p\to1^+}\PAE\INTO\RHOPB\FUNUPB\DX\PAD\mI\INTO\RHOB\FUNUB\DX.
\end{align}
Taking $\FUNUB$ as a test function in \eqref{COND1}, we obtain
\begin{align*}
    \INTO\CAMZB\cdot\nabla\FUNUB\DX+\INTO\MOE\nabla\FUNUB\MOD^{q}\DX=\INTO\RHOB\FUNUB\DX.
\end{align*}
Since $\CAMZB\cdot\nabla\FUNUB=\MOE\nabla\FUNUB\MOD$ almost everywhere in $\Omega$, it follows that
\begin{align*}
    \INTO\MOE\nabla\FUNUB\MOD\DX+\INTO\MOE\nabla\FUNUB\MOD^{q}\DX=\INTO\RHOB\FUNUB\DX.
\end{align*}
Substituting this expression into \eqref{NTG}, we obtain
\begin{align*}
    0<\alpha\mI\INTO\MOE\nabla\FUNUB\MOD\DX+\INTO\MOE\nabla\FUNUB\MOD^{q}\DX.
\end{align*}
Since $\FUNUB\in\ESPWUQZ\POmega$, we may apply Hölder’s Inequality to obtain $C>0$ such that
\begin{align*}
    \INTO\MOE\nabla\FUNUB\MOD\DX\mI\MOE\Omega\MOD^{\frac{q-1}{q}}\NOE\nabla\FUNUB\NODLQ
\end{align*}
and, consequently, 
\begin{align*}
    \INTO\MOE\nabla\FUNUB\MOD\DX+\INTO\MOE\nabla\FUNUB\MOD^q\mI C\NOE\nabla\FUNUB\NODLQ+\NOE\nabla\FUNUB\NODLQ^q.
\end{align*}
Finally, since $\NOE\FUNUB\NODWUQZ=\NOE\nabla\FUNUB\NODLQ$, we get
\begin{align*}
    0<\alpha\mI C\NOE\FUNUB\NODWUQZ+\NOE\FUNUB\NODWUQZ^q
\end{align*}
and then $\FUNUB\neq 0$.
\end{proof}
\section{Behavior of solutions as $\beta\to0^+$}\label{S04}

In this section we are going to prove Theorem \ref{THEO2}. To this end, for each $\beta>0$, we consider the functional $I_\beta:\ESPWUQZ\POmega\to\CR$ given by
\begin{align*}
    I_\beta(u):=\INTO\MOE\nabla u\MOD+\dfrac{1}{q}\INTO\MOE\nabla u\MOD^q\DX-\INTO F_\beta(u)\DX.
\end{align*}

Our goal now is to compare the limit energy $I_\beta(\FUNUB)$ with the approximate energy $I_{p,\beta}(\FUNUPB)$, showing that the energy of the limit solution coincides with the limit of the energies of the approximating solutions.

Since $\FUNUB$ is a solution of the limit problem, that is, the Problem \eqref{P}, by the first and second equations in \eqref{HIPZB} and by Green’s Formula, we obtain
\begin{align*}
    -\INTO\FUNUB\hspace{0,1cm}\DIVE(\CAMZB)\DX=
    \INTO\CAMZB\cdot\nabla\FUNUB\DX=
    \INTO\MOE\nabla\FUNUB\MOD\DX
\end{align*}
and, consequently,
\begin{align*}
    \INTO\MOE\nabla\FUNUB\MOD\DX+\INTO\MOE\nabla\FUNUB\MOD^q\DX=\INTO\RHOB\FUNUB\DX.
\end{align*}
On the other hand, taking $\FUNUPB$ as a test function in Problem \eqref{PP}, we obtain
\begin{align*}
    \INTO\MOE\nabla\FUNUPB\MOD^p\DX+\INTO\MOE\nabla\FUNUPB\MOD^q\DX=\INTO\RHOPB\FUNUPB\DX.
\end{align*}
Since, up to a subsequence, that the convergences given in expression \eqref{CONVER1} and \eqref{E326} hold as $p\to1^+$, it follows that
\begin{align*}
    \INTO\RHOB\FUNUB\DX=\INTO\RHOPB\FUNUPB\DX+o_p(1).
\end{align*}
Consequently, from the three expressions obtained above, we can conclude that
\begin{align*}
    \INTO\MOE\nabla\FUNUB\MOD\DX+\INTO\MOE\nabla\FUNUB\MOD^q\DX=\INTO\MOE\nabla\FUNUPB\MOD^p\DX+\INTO\MOE\nabla\FUNUPB\MOD^q\DX+o_p(1).
\end{align*}

Moreover, by the lower semicontinuity of the norm with respect to weak convergence, we obtain that
\begin{align*}
    \INTO\MOE\nabla\FUNUB\MOD\DX=\INTO\MOE\nabla\FUNUPB\MOD^p\DX+o_p(1)
\end{align*}
and
\begin{align*}
    \dfrac{1}{q}\INTO\MOE\nabla\FUNUB\MOD\DX=\dfrac{1}{q}\INTO\MOE\nabla\FUNUPB\MOD^q\DX+o_p(1).
\end{align*}
From \eqref{E326} and \eqref{E328}, it follows that
\begin{align*}
    \INTO F_\beta(\FUNUB)\DX=\INTO F_\beta(\FUNUPB)\DX+o_p(1).
\end{align*}

Therefore, by gathering the previous estimates and using the definition of the energy functional $I_\beta$, we conclude that
\begin{align}\label{E44}
    I_\beta(\FUNUB)=I_{p,\beta}(\FUNUPB)+o_p(1).
\end{align}
This equality shows that the energy of the limit solution coincides with the limit of the energies of the approximating solutions.
\newline 

Since we are interested in the behaviour of $\FUNUB$ as $\beta\to0^+$, we shall assume throughout this section that $\beta\in(0,\beta_0)$, for some $\beta_0 > 0$ fixed.
\begin{LEM}\label{LEM11}
    The family $(\FUNUB)_{\beta\in(0,\beta_0)}$ is bounded in $\ESPWUQZ\POmega$.
\end{LEM}
\begin{proof}
Using the same reasoning employed \cite{PimentaJunior2024}[Lemma 4.1], we assert that
\begin{align}\label{E1LEM11}
    I_\beta(\FUNUB)\mI I_{\beta_0}(\FUNUBZ)=:C
\end{align}
for all $\beta\in(0,\beta_0)$. Note that, since $\FUNUB$ is a critical point of $I_\beta$, it follws that
\begin{align*}
    0\in\partial I_\beta(\FUNUB)
\end{align*}
and, consequently, there exists $\RHOB\in\partial F_\beta(\FUNUB)$ such that
\begin{align*}
    Q^\prime(\FUNUB)-\RHOB=0,
\end{align*}
where $Q(u)=\INTO\MOE\nabla u\MOD\DX+\frac{1}{q}\INTO\MOE\nabla u\MOD^q\DX$. Taking the dual product of the above expression with $\FUNUB$, we obtain
\begin{align*}
    \POE Q^\prime(\FUNUB),\FUNUB\POD
    &=\INTO\RHOB\FUNUB\DX\\
    &=\INTO\MOE\nabla\FUNUB\MOD\DX+\INTO\MOE\nabla\FUNUB\MOD^q\DX
\end{align*}
and, consequently,
\begin{align}\label{E2LEM11}
    \INTO\MOE\nabla\FUNUB\MOD\DX+\INTO\MOE\nabla\FUNUB\MOD^q\DX=\INTO\RHOB\FUNUB\DX.
\end{align}
By the definitions of $I_\beta$ and $Q$, we have
\begin{align*}
    I_\beta(\FUNUB)=Q(\FUNUB)-\INTO F_\beta(\FUNUB)\DX
\end{align*}
and, by using the previous identities, we obtain
\begin{align*}
    I_\beta(\FUNUB)
    &=Q(\FUNUB)-\dfrac{1}{s}\POE Q^\prime(\FUNUB)-\RHOB,\FUNUB\POD\\
    &=Q(\FUNUB)-\dfrac{1}{s}\POE Q^\prime(\FUNUB),\FUNUB\POD+\dfrac{1}{s}\INTO\RHOB\FUNUB\DX\\
    &=\INTO\MOE\nabla\FUNUB\MOD\DX+\dfrac{1}{q}\INTO\MOE\nabla\FUNUB\MOD^q\DX-\INTO F_\beta(\FUNUB)\DX-\dfrac{1}{s}\INTO\MOE\nabla\FUNUB\MOD\DX\\
    &\hspace{0,5cm}-\dfrac{1}{s}\INTO\MOE\nabla\FUNUB\MOD^q\DX+\dfrac{1}{s}\INTO\RHOB\FUNUB\DX\\
    &=\PAE 1- \dfrac{1}{s}\PAD\INTO\MOE\nabla\FUNUB\MOD\DX+\PAE\dfrac{1}{q}-\dfrac{1}{s}\PAD\INTO\MOE\nabla\FUNUB\MOD^q\DX\\
    &\hspace{0.5cm}+
    \INTO\PAE\dfrac{1}{s}\RHOB\FUNUB-F_\beta(\FUNUB)\PAD\DX.
\end{align*}
As in \eqref{eqbounded31}, we know that $\frac{1}{s}\RHOB\FUNUB-F_\beta(\FUNUB)\MI 0$, which implies that
\begin{align*}
    I_\beta(\FUNUB)\MI\PAE 1-\dfrac{1}{s}\PAD\INTO\MOE\nabla\FUNUB\MOD\DX+\PAE\dfrac{1}{q}-\dfrac{1}{s}\PAD\INTO\MOE\nabla\FUNUB\MOD^q\DX.
\end{align*}
Since $1<q<s$, it follows from that last expression and from \eqref{E1LEM11} that $(\FUNUB)_{\beta\in(0,\beta_0)}$ is bounded in $\ESPWUQZ\POmega$.
\end{proof}

As a consequence of Lemma \ref{LEM11}, since the embedding $\ESPWUQZ\POmega\hookrightarrow \ESPLR\POmega$ is compact for all $1\mI r<q^*$, there exists a function $\FUNUZ\in\ESPWUQZ\POmega$ such that, up to a subsequence, as $\beta\to0^+$
\begin{align}
    \label{CONVUZ1}\FUNUB&\rightharpoonup\FUNUZ\quad\mbox{ in }\ESPWUQZ\POmega,\\
    \label{CONVUZ2}\FUNUB&\to\FUNUZ\quad\mbox{ in }L^r\POmega,\mbox{ with }1 \leq r < q^*,\\
    \label{CONVUZ3}\FUNUB&\to\FUNUZ\quad\mbox{ a.e. in } \Omega.
\end{align}

Moreover, since $(\FUNUB)_{\beta\in(0,\beta_0)}$ is uniformly bounded in $\ESPWUQZ\POmega$ and the convergence given in the Expression \eqref{E329} hold, the family $(\RHOB)_{\beta\in(0,\beta_0)}$ is bounded in $L^{\frac{s}{s-1}}\POmega$. Hence, using the same reasoning employed in Lemma \ref{LEM8}, we can deduce that there exists $\RHOZ\in L^{\frac{s}{s-1}}\POmega$ such that, up to a subsequence, as $\beta\to 0^+$
\begin{align}
    \label{CONVRHO1}\RHOB&\rightharpoonup\RHOZ\quad\mbox{ in }L^{\frac{s}{s-1}},\\
    \label{CONVRHO2}\RHOB&\to\RHOZ\quad\mbox{ a.e. in } \Omega.
\end{align}
Furthermore, by the definition of $\RHOB$, we have
\begin{align*}
    0\mI\RHOZ(x)\mI\MOE\FUNUZ(x)\MOD^{q-1}
\end{align*}
for almost every $x\in\Omega$.

Finally, we now analyze the behavior of the family of vector fields $(\CAMZB)_{\beta\in(0,\beta_0)}$. By construction, these fields satisfy the uniform bound
\begin{align*}
    \NOE\CAMZB\NODLINF\mI 1
\end{align*}
for all $\beta\in(0,\beta_0)$ and, consequently, the sequence $(\CAMZB)_{\beta\in(0,\beta_0)}$ is bounded in $\ESPLINF(\Omega,\CR^N)$. Since $\ESPLINF(\Omega,\CR^N)$ is the dual space of $L^1(\Omega,\CR^N)$, there exist a vector field $\CAMZZ\in\ESPLINF(\Omega,\CR^N)$ and a subsequence, still denoted by $(\CAMZB)$, such that
\begin{align}\label{CONVRHOZ1}
    \CAMZB\overset{*}{\rightharpoonup}\CAMZZ
\end{align}
in $\ESPLINF(\Omega,\CR^N)$, as $\beta\to 0^+$. This weak-$*$ convergence in $\ESPLINF$ means that
\begin{align*}
    \INTO\CAMZB\cdot\psi\DX\to\INTO\CAMZZ\cdot\psi\DX
\end{align*}
for all $\psi\in L^1(\Omega,\CR^N)$, as $\beta\to0^+$. In particular, for $\varphi\in\ESPCINFC\POmega$,
\begin{align*}
    \INTO\CAMZB\cdot\nabla\varphi\DX\to\INTO\CAMZZ\cdot\nabla\varphi\DX
\end{align*}
as $\beta\to0^+$. Then
\begin{align*}
    \DIVE(\CAMZB)\to\DIVE(\CAMZZ)\quad\mbox{ in }\mathcal{D}^\prime(\Omega).
\end{align*}

Combining this convergence with the fact that $-\DIVE(\CAMZB)=\RHOB$ in $\mathcal{D}^\prime\POmega$ and using the convergences given in expressions \eqref{CONVRHO1} and \eqref{CONVRHO2}, we obtain
\begin{align}\label{COND1UZ}
    -\DIVE(\CAMZZ)=\RHOZ\quad\mbox{ in }\mathcal{D}^\prime\POmega.
\end{align}

With these considerations, we can establish the following result, which characterizes the relationship between the limit function $\FUNUZ$ and the vector field $\CAMZZ$.
\begin{LEM}
The function $\FUNUZ$ and the vector field $\CAMZZ$ satisfy
\begin{align}\label{COND2UZ}
    \CAMZZ\cdot\nabla\FUNUZ=\MOE\nabla\FUNUZ\MOD\quad\mbox{ a.e. in }\Omega.
\end{align}
\end{LEM}
\begin{proof}
To prove this result, we will employ the same reasoning used in the proof of Lemma \ref{LEM10}. Note that, by passing to the limit as $\beta\to0^+$ in Expression \eqref{E01LEM8} and using the convergence results obtained previously, we deduce that 
\begin{align*}
    \CAMZZ\cdot\nabla\FUNUZ\mI\MOE\nabla\FUNUZ\MOD\quad\mbox{ a.e. in }\Omega.
\end{align*}
We now prove the reverse inequality. It suffices to show that for every nonnegative $\varphi\in\ESPCUC\POmega$,
\begin{align*}
    \INTO\varphi\MOE\nabla\FUNUZ\MOD\DX \leq \INTO\varphi\CAMZZ\cdot\nabla\FUNUZ\DX.
\end{align*}
By taking $\FUNUB\varphi$ as test function in Problem \eqref{HIPZB} (after a regularizing process), we obtain
\begin{align*}
  \int_\Omega \varphi |\nabla u_\beta| \DX = &- \int_\Omega u_\beta {\bf z_\beta}\cdot \nabla\varphi \DX
     - \int_\Omega \varphi |\nabla u_\beta|^q \DX\\
     &  - \int_\Omega u_\beta |\nabla u_\beta|^{q-2}\nabla u_\beta \cdot \nabla \varphi \DX + \int_\Omega \rho_\beta \varphi u_\beta \DX.
\end{align*}
From the last expression and \eqref{E01LEM10},
$$
\int_\Omega \varphi |\nabla u_\beta| \DX = \int_\Omega \varphi {\bf z_\beta}\cdot \nabla u_\beta \DX.
$$
Hence, from lower semicontinuity, \eqref{CONVUZ1} and \eqref{CONVRHOZ1}, it follows that
\begin{align*}
    \INTO\varphi\MOE\nabla\FUNUZ\MOD\DX \leq \INTO\varphi\CAMZZ\cdot\nabla\FUNUZ\DX.
\end{align*}
Since $\varphi\MI 0$ is arbitrary, we conclude that
\begin{align*}
    \CAMZZ\cdot\nabla\FUNUZ\MI\MOE\nabla\FUNUZ\MOD\quad\mbox{ a.e. in }\Omega.
\end{align*}
\end{proof}

The next step in the proof of Theorem \ref{THEO2} is to establish the existence of constants $\mu,\beta_0>0$ such that
\begin{align}\label{E420}
    \MOE\{x\in\Omega;\ \FUNUB(x)>\beta\}\MOD\MI\mu
\end{align}
for all $\beta\in(0,\beta_0)$, ensuring, roughly speaking, that the solutions do not concentrate entirely below the limiting level $\beta$. 

From the variational construction of the approximating solutions $\FUNUPB$, there exists a constant $\alpha>0$, independent of $p\in(1,\PB)$, such that
\begin{align}\label{eqmeasure1}
    \alpha+o_p(1)\mI c_{p,\beta}.
\end{align}
Since $\FUNUPB$ is a weak solution of the Problem \eqref{PP}, we have
\begin{align*}
    \INTO\MOE\nabla\FUNUPB\MOD^q\DX+\INTO\MOE\nabla\FUNUPB\MOD^p\DX=\INTO\RHOPB\FUNUPB\DX.
\end{align*}

On the other hand, using the definition of the functional and the properties of the subdifferential, we obtain
\begin{align}\label{eqmeasure2}
    c_{p,\beta}\mI\dfrac{1}{p}\INTO\RHOPB\FUNUPB\DX+o_p(1).
\end{align}
Combining \eqref{eqmeasure1} and \eqref{eqmeasure2}, 
\begin{align*}
    0<\alpha+o_p(1)\mI\dfrac{1}{p}\INTO\RHOPB\FUNUPB\DX+o_p(1),
\end{align*}
where $\alpha$ is independent of $\beta$ and $p$. Recall that, since $\RHOPB$ satisfies the equations given in \eqref{DEFINRHO}, we have
\begin{align*}
    0\mI\RHOPB(x)\FUNUPB(x)\mI\beta^s+\FUNUPB^s(x)\mathcal{X}_{\{\FUNUPB>\beta\}}(x)
\end{align*}
for almost every $x\in\Omega$. Integrating this inequality over $\Omega$, we obtain
\begin{align*}
    \INTO\RHOPB\FUNUPB\DX\mI\beta^s\MOE\Omega\MOD+\int_{\{\FUNUPB>\beta\}}\FUNUPB^s\DX.
\end{align*}
Thus, combining this estimate with the previously obtained inequality, it follows that
\begin{align*}
    \alpha\mI\dfrac{\beta^s}{p}\MOE\Omega\MOD+\dfrac{1}{p}\int_{\{u_{p,\beta}>\beta\}} u_{p,\beta}^s \DX + o_p(1).
\end{align*}

Now, using the convergences given in expressions \eqref{E326} and \eqref{E327}, we can pass to the limit as $p\to1^+$, ensuring, by the Dominated Convergence Theorem, that
\begin{align*}
    \int_{\{\FUNUPB>\beta\}}\FUNUPB^s\DX\to\int_{\{u_\beta>\beta\}}\FUNUB^s\DX
\end{align*}
and, consequently,
\begin{align}\label{E423}
    \alpha\mI\beta^s\MOE\Omega\MOD+\int_{\{u_\beta>\beta\}}\FUNUB^s\DX.
\end{align}

From Lemma \ref{LEM7} we know that the sequence $(\FUNUPB)$ is uniformly bounded in $\ESPLINF\POmega$. Passing to the limit as $p\to 1^+$ and, using the convergence given in Expression \eqref{E327}, we can conclude that $\FUNUB\in\ESPLINF\POmega$ and 
\begin{align*}
    \NOE\FUNUB\NODLINF\mI\CINF.
\end{align*}
Consequently, there exists a constant $M>0$ such that $0\mI\FUNUB(x)\mI M$ for almost every $x\in\Omega$. Hence
\begin{align*}
    \int_{\{\FUNUPB>\beta\}}\FUNUB^s\DX\mI M^s\MOE\{\FUNUB>\beta\}\MOD.
\end{align*}
Combining the previous inequalities, we obtain
\begin{align*}
    \alpha\mI\beta^s\MOE\Omega\MOD+M^s\MOE\{\FUNUB>\beta\}\MOD.
\end{align*}
Choosing $\beta_0>0$ sufficiently small and defining $\mu=\dfrac{\alpha-\beta_0^s\MOE\Omega\MOD}{M^s}$, we conclude that
\begin{align*}
    \MOE\{x\in\Omega;\ \FUNUB(x)>\beta\}\MOD\MI\mu
\end{align*}
for every $\beta\in(0,\beta_0)$.
\newline
This also implies that $u_0 \not\equiv 0$.

Finally, it remains to show that $\FUNUZ\in\ESPWUQZ\POmega$ is indeed a solution of Problem \eqref{P0}. To that, recall that the functions $\RHOB$ satisfy the pointwise relation
\begin{align*}
    \RHOB(x)\in\QUE\underline{f}_\beta\PAE\FUNUB(x)\PAD,\overline{f}_\beta\PAE\FUNUB(x)\PAD\QUD
\end{align*}
for almost every $x\in\Omega$. Using Expression \eqref{DEFINRHO} and the convergences given in expressions \eqref{CONVUZ2} and \eqref{CONVRHO2}, we conclude that
\begin{align*}
    \RHOZ(x)=\FUNUZ^{s-1}(x)
\end{align*}
for almost every $x\in\Omega$. Therefore, the pair $(\FUNUZ,\RHOZ)$ satisfies the limit equation in the weak sense. In particular, $\FUNUZ$ is a nonnegative and nontrivial solution of the limit Problem \eqref{P0}.


\begin{thebibliography}{00}
\bibitem{AndreuBallesterCasellesMazon1}
F. Andreu, C. Ballester, V. Caselles and 
\newblock{\href{https://doi.org/10.1016/S0764-4442(00)01729-8}{Minimizing total variation flow}},
\newblock\emph{Comptes Rendus de l'Académie des Sciences, Series I} \textbf{331} (2000), no. 11, 867--872.

\bibitem{AndreuBallesterCasellesMazon2}
F. Andreu, C. Ballester, V. Caselles and J. M. Mazón,
\newblock{\href{https://doi.org/10.1006/jfan.2000.3698}{The Dirichlet problem for the total variation flow}},
\newblock\emph{Journal of Functional Analysis} \textbf{180} (2001), no. 2, 347--403.

\bibitem{AndreuBallesterCasellesMazon3}
F. Andreu, C. Ballester, V. Caselles and J. M. Mazón,
\newblock{\href{https://projecteuclid.org/journals/differential-and-integral-equations/volume-14/issue-3/Minimizing-total-variation-flow/die/1356123331.short}{The Dirichlet problem for the total variation flow}},
\newblock\emph{Differential Integral Equations} \textbf{14} (2000), no. 3, 321--360.

\bibitem{Anzellotti1983}
G. Anzellotti,
\newblock{\href{http://doi.org/10.1007/BF01781073}{Pairings between measures and bounded functions and compensated compactness}},
\newblock\emph{Annali di Matematica Pura ed Applicata} \textbf{135} (1983), 293--318.

\bibitem{AparicioOlivaPetitta2025}
A. J. M. Aparicio, F. Oliva and F. Petitta,
\newblock{\href{https://doi.org/10.1007/s00526-025-03102-6}{The Sattinger iteration method for $1$-Laplace type problems and its application to concave-convex nonlinearities}},
\newblock\emph{Calculus of Variations and Partial Differential Equations} \textbf{64} (2025), 251.

\bibitem{ArcoyaCalahorrano1994}
D. Arcoya and M. Calahorrano,
\newblock{\href{https://doi.org/10.1006/jmaa.1994.1406}{Some discontinuous problems with a quasilinear operator}},
\newblock\emph{Journal of Mathematical Analysis and Applications} \textbf{187} (1994), 1059--1072.

\bibitem{AttouchButtazzoMichaille2014}
H. Attouch, G. Buttazzo and G. Michaille,
\newblock{\href{https://epubs.siam.org/doi/book/10.1137/1.9781611973488}{Variational analysis in Sobolev and BV spaces: applications to PDEs and optimization}},
\newblock Society for Industrial and Applied Mathematics, 2014.

\bibitem{BarileFigueiredo2015}
S. Barile and G. M. Figueiredo,
\newblock{\href{https://doi.org/10.1016/j.na.2014.11.002}{Some classes of eigenvalues problems for generalized 
$p\&q$-Laplacian type operators on bounded domains}},
\newblock\emph{Nonlinear Analysis: Theory, Methods \& Applications} \textbf{119} (2015), 457--468.

\bibitem{BocardoMurat}
L. Boccardo and F. Murat,
\newblock{\href{https://doi.org/10.1016/0362-546X(92)90023-8}{Almost everywhere convergence of the gradients of solutions to elliptic and parabolic equations}},
\newblock\emph{Nonlinear Analysis: Theory, Methods \& Applications} \textbf{19} (1992), no. 6, 581--597.

\bibitem{Brezis2011}
H. Brézis,
\newblock{\href{http://doi.org/10.1007/978-0-387-70914-7}{Functional Analysis, Sobolev Spaces and Partial Differential Equations}},
\newblock Springer, New York, 2011.

\bibitem{Chang1981}
K. C. Chang,
\newblock{\href{https://doi.org/10.1016/0022-247X(81)90095-0}{Variational methods for nondifferentiable functionals and their applications to partial differential equations}},
\newblock\emph{Transactions of the American Mathematical Society} \textbf{80} (1981), no. 1, 102--129.

\bibitem{ChataPimentaSegura2023}
J. C. O. Chata, M. T. O. Pimenta and S. Segura de León,
\newblock{\href{https://doi.org/10.1016/j.jmaa.2023.127149}{Existence of solutions to a $1$-Laplacian problem with a concave-convex nonlinearity}},
\newblock\emph{Journal of Mathematical Analysis and Applications} \textbf{525} (2023), no. 2, 127149.

\bibitem{ChavesErcoleMiyagaki2015}
M. F. Chaves, G. Ercole and O. H. Miyagaki,
\newblock{\href{https://doi.org/10.1016/j.na.2014.11.010}{Existence of a nontrivial solution for the $(p,q)$-Laplacian problems in $\mathbb{R}^N$ without the Ambrosetti--Rabinowitz condition}},
\newblock\emph{Nonlinear Analysis: Theory, Methods \& Applications} \textbf{114} (2015), 133--141.

\bibitem{Clarke1975}
F. H. Clarke,
\newblock{\href{https://doi.org/10.1090/S0002-9947-1975-0367131-6}{Generalized gradients and applications}},
\newblock\emph{Transactions of the American Mathematical Society} \textbf{205} (1975), 247--262.

\bibitem{Clarke1983}
F. H. Clarke,
\newblock{\href{https://epubs.siam.org/doi/book/10.1137/1.9781611971309}{Optimization and Nonsmooth Analysis}},
\newblock John Wiley \& Sons, New York, 1983.

\bibitem{Demengel}
F. Demengel,
\newblock{\href{https://www.numdam.org/item/COCV_1999__4__667_0/?source=COCV_1999__4__559_0}{On some nonlinear partial differential equations involving the $1-$Laplacian and critical Sobolev exponent}},
\newblock\emph{ESAIM: Control, Optimisation and Calculus of Variations} \textbf{4} (1999), 667--686.

\bibitem{FariaMiyagakiMotreanu2014}
L. F. O. Faria, O. H. Miyagaki and D. Motreanu,
\newblock{\href{https://www.cambridge.org/core/journals/proceedings-of-the-edinburgh-mathematical-society/article/comparison-and-positive-solutions-for-problems-with-the-p-qlaplacian-and-a-convection-term/C87AFE487E0E6BCE151E9C85B11AA5B9}{Comparison and positive solutions for problems with the $(p,q)$-Laplacian and a convection term}},
\newblock\emph{Proceedings of the Edinburgh Mathematical Society} \textbf{57} (2014), no. 3, 687--698.

\bibitem{FigueiredoPimentaWinkert2025}
G. M. Figueiredo, M. T. O. Pimenta and P. Winkert,
\newblock{\href{https://doi.org/10.1016/j.na.2024.113677}{The asymptotic behavior of constant sign and nodal solutions of $(p,q)$-Laplacian problems as $p \to 1$}},
\newblock\emph{Nonlinear Analysis} \textbf{251} (2025), 113677.

\bibitem{GilbargTrudinger1977}
D. Gilbarg and N. S. Trudinger,
\newblock{\href{https://doi.org/10.1007/978-3-642-61798-0}{Elliptic Partial Differential Equations of Second Order}},
\newblock Springer-Verlag, Berlin, 1977.

\bibitem{GrossinhoTersian2001}
M. R. Grossinho and S. A. Tersian,
\newblock{\href{https://books.google.com.br/books/about/An_Introduction_to_Minimax_Theorems_and.html?id=cxJY8yubIGYC&redir_esc=y}{An Introduction to Minimax Theorems and Their Applications to Differential Equations}},
\newblock Springer, Boston, 2001.

\bibitem{Kavian}
O. Kavian,
\newblock{\href{https://www.researchgate.net/publication/265334045_Introduction_a_la_theorie_des_points_critiques_et_applications_aux_problemes_elliptiques}{Introduction à la théorie des points critiques et applications aux problèmes elliptiques}},
\newblock Springer, Paris, 1993.

\bibitem{MolinoSegura2018}
A. Molino and S. Segura de León,
\newblock{\href{https://doi.org/10.1016/j.na.2017.11.006}{Elliptic equations involving the $1$-Laplacian and a subcritical source term}},
\newblock\emph{Nonlinear Analysis} \textbf{168} (2018), 50--66.

\bibitem{MotreanuWinkert2019}
D. Motreanu and P. Winkert,
\newblock{\href{https://doi.org/10.1016/j.aml.2019.03.023}{Existence and asymptotic properties for quasilinear elliptic equations with gradient dependence}},
\newblock\emph{Applied Mathematics Letters} \textbf{95} (2019), 78--84.

\bibitem{PapageorgiouWinkert2021}
N. S. Papageorgiou and P. Winkert,
\newblock{\href{https://doi.org/10.1007/s00009-021-01780-y}{Singular Dirichlet $(p,q)$-equations}},
\newblock\emph{Mediterranean Journal of Mathematics} \textbf{18} (2021), 141.

\bibitem{PimentaCarranzaFigueiredo2024}
M. T. O. Pimenta, Y. B. C. Carranza and G. M. Figueiredo,
\newblock{\href{https://doi.org/10.1007/s11784-024-01138-3}{Existence results for the $1$-Laplacian problem with a critical concave--convex nonlinearity}},
\newblock\emph{Journal of Fixed Point Theory and Applications} \textbf{26} (2024), 49.

\bibitem{PimentaJunior2024}
M. T. O. Pimenta, G. C. G. dos Santos and J. R. Santos Júnior,
\newblock{\href{https://doi.org/10.1017/prm.2022.86}{On a quasilinear elliptic problem involving the $1$-Laplacian operator and a discontinuous nonlinearity}},
\newblock\emph{Proceedings of the Royal Society of Edinburgh} \textbf{154} (2024), no. 1, 33--59.

\bibitem{PimentaSantosStapenhorst2024}
M. T. O. Pimenta, J. R. Santos Júnior and M. F. Stapenhorst,
\newblock{\href{https://doi.org/10.1016/j.jde.2024.05.014}{A $1$-Laplacian equation with critical and discontinuous nonlinearities}},
\newblock\emph{Journal of Differential Equations} \textbf{402} (2024), 463--494.

\bibitem{SantosFigueiredoPimenta2022}
G. C. G. Santos, G. M. Figueiredo and M. T. O. Pimenta,
\newblock{\href{https://doi.org/10.1007/s12220-022-00881-8}{Multiple ordered solutions for a class of problems involving the $1$-Laplacian operator}},
\newblock\emph{The Journal of Geometric Analysis} \textbf{32} (2022), 140.

\end{thebibliography}
\end{document}